\documentclass[12pt]{article}
\usepackage{amssymb}
\usepackage{amsfonts}
\usepackage{amscd}
\usepackage{mathrsfs}
\usepackage[reqno]{amsmath}
\usepackage{amsthm,upref,amscd}
\textwidth 155mm \textheight 230mm \theoremstyle{plain}

\newtheorem{Thm}{Theorem}[section]

\newtheorem{Lem}[Thm]{Lemma}
\newtheorem{Cor}[Thm]{Corollary}

\newtheorem{Rem}[Thm]{Remark}
\newtheorem{Def}[Thm]{Definition}

\theoremstyle{definition}

\numberwithin{equation}{section}

\def\R{\mathbb{R}}
\def\N{\mathbb{N}}

\def\va{\varepsilon}

\def\di{\displaystyle\int}
\def\df{\displaystyle\frac}

\begin{document}
\thispagestyle{empty}

\title{Infinitely many solutions of a class of elliptic equations with variable exponent}
\author{Chang-Mu Chu$^{\rm 1}$\thanks{Corresponding
author. Email address: gzmychuchangmu@sina.com.},
\  Haidong Liu$^{\rm 2}$\vspace{1mm}\\
{$^{\rm 1}$\small School of Data Science and Information Engineering, Guizhou Minzu University}\\
{\small Guizhou 550025, P.R. China}\\
{$^{\rm 2}$\small College of Mathematics, Physics and Information Engineering, Jiaxing University}\\
{\small Zhejiang 314001, P.R. China}}

\date{}

\maketitle

\begin{abstract}
This paper is concerned with the $p(x)$-Laplacian equation of the form
\begin{equation}\label{eq0.1}
\left\{\begin{array}{ll}
-\Delta_{p(x)} u=Q(x)|u|^{r(x)-2}u, &\mbox{in}\ \Omega,\\
u=0, &\mbox{on}\ \partial \Omega,
\end{array}\right.
\end{equation}
where $\Omega\subset\R^N$ is a smooth bounded domain,
$1<p^-=\min_{x\in\overline{\Omega}}p(x)\leq p(x)\leq\max_{x\in\overline{\Omega}}p(x)=p^+<N$,
$1\leq r(x)<p^{*}(x)=\frac{Np(x)}{N-p(x)}$, $r^-=\min_{x\in \overline{\Omega}}r(x)<p^-$,
$r^+=\max_{x\in\overline{\Omega}}r(x)>p^+$ and $Q: \overline{\Omega}\to\R$ is a nonnegative
continuous function. We prove that \eqref{eq0.1} has infinitely many small solutions and infinitely
many large solutions by using  the Clark's theorem and the symmetric mountain pass lemma.

\end{abstract}

\section{Introduction and main results}

\hspace*{\parindent}In recent years, the following nonlinear elliptic equation
\begin{equation}\label{eq1.1}
\left\{\begin{array}{ll}
-\Delta_{p(x)} u=f(x,u), &\mbox{in}\ \Omega,\\
u=0, &\mbox{on}\ \partial \Omega
\end{array}\right.
\end{equation}
has received considerable attention due to the fact that it can be applied to fluid mechanics and the
field of image processing (see \cite{CLR, RM}), where $\Omega\subset  \mathbb{R}^{N}$ is a smooth
bounded domain, $p: \overline{\Omega} \rightarrow \R$ is a continuous function satisfying
$1<p^-=\min_{x\in \overline{\Omega}}p(x)\leq p(x)\leq\max_{x\in \overline{\Omega}}p(x)=p^+<N $ and
$f: \overline{\Omega}\times\R \rightarrow \R$ is a suitable function.

In 2003, Fan and Zhang in \cite{FZ} gave several sufficient conditions for the existence and multiplicity
of nontrivial solutions for problem \eqref{eq1.1}. These conditions include either the sublinear growth
condition
$$|f(x,t)|\leq C\left(1+|t|^\beta\right),\ \ \mbox{for}  \ x\in\Omega\ \mbox{and}\ t\in\R$$
or Ambrosetti-Rabinowitz type superlinear  condition ($(AR)$-condition, for short)
$$f(x,t)t\geq \theta F(x,t)>0,\ \ \text{for}\ x\in\Omega\ \mbox{and}\ |t| \ \text{sufficiently large},$$
where $C>0$, $1\leq\beta<p^-$, $\theta>p^+$ and $F(x,t)=\int_0^tf(x,s)\,ds$. Subsequently, Chabrowski
and Fu in \cite{CF} discussed problem \eqref{eq1.1} in a more general setting than that in \cite{FZ}. It is
well known that $(AR)$-condition is important to guarantee the boundedness of Palais-Smale sequence
of the Euler-Lagrange functional which plays a crucial pole in applying the critical point theory. However,
it excludes many cases of nonlinearity (see \cite{JF, RV, TF, YZ, ZA, ZZ}). In fact, either the uniform
superlinear growth condition or the uniform sublinear growth condition was still imposed on $f(x,t)$. In
addition, some papers discussed problem \eqref{eq1.1} with concave-convex nonlinearities (see
\cite{GZZ, MOY, NT, YW}).

For the case $f(x,t)=Q(x)|t|^{r(x)-2}t$, problem \eqref{eq1.1} reduces to
\begin{equation}\label{eq1.2}
\left\{\begin{array}{ll}
-\Delta_{p(x)} u=Q(x)|u|^{r(x)-2}u, &\mbox{in}\ \Omega,\\
u=0, &\mbox{on}\ \partial \Omega,
\end{array}\right.
\end{equation}
where $Q, r: \overline{\Omega} \rightarrow \R$ are nonnegative continuous functions. The sets
$\Omega_0=\{x\in \Omega\,|\,r(x)=p(x)\}$, $\Omega_-=\{x\in \Omega\,|\,r(x)<p(x)\}$ and
$\Omega_+=\{x\in \Omega\,|\,r(x)>p(x)\}$ can have positive measure at the same time. This situation
is new and closely related to the existence of variable exponents since we can't meet such a phenomenon
in the constant exponent case. Mih$\check{a}$ilescu and R$\check{a}$dulescu in \cite{MR} have
considered problem \eqref{eq1.2} with $Q(x)\equiv \lambda$ under the basic assumption
$1<r^-=\min_{x\in \overline{\Omega}}r(x)<p^-<r^+=\max_{x\in \overline{\Omega}}r(x)$ and proved that
there exists $\lambda_0>0$ such that any $\lambda\in (0,\lambda_0)$ is an eigenvalue for problem
\eqref{eq1.2}. Subsequently, Fan in \cite{FX} extended the main results of \cite{MR} in the case
$\Omega=\Omega_{-}$ (but $r^+<p^-$ does not hold) and in the case $\Omega=\Omega_{+}$ (but
$r^->p^+$ does not hold), respectively. Their results implied that for any positive constant $C>0$ there
exists $u_0\in W^{1,p(x)}_0(\Omega)$ such that
$$C\di_{\Omega}|u_0|^{r(x)}dx\geq \di_{\Omega}|\nabla u_0|^{p(x)}dx.$$
Therefore, we have to overcome new difficulties in dealing with \eqref{eq1.2}.

Different from the concave-convex nonlinearities, the main feature of problem \eqref{eq1.2} is that
$Q(x)|t|^{r(x)}$ has both local superlinear growth and local sublinear growth. Due to this, it is difficult to
prove the boundedness of Palais-Smale sequence of the Euler-Lagrange functional. To the best of our
knowledge, we only realize that Aouaoui \cite{AS} obtainded at least three nontrivial solutions of problem \eqref{eq1.2} with $\Omega=\R^N$ by perturbation method. In the current paper, we are concerned with
the existence of infinitely many small solutions and infinitely many large solutions for problem \eqref{eq1.2}
under the assumption $1\leq r(x)<p^*(x)=\frac{Np(x)}{N-p(x)}$ and  $r^-<p^-\leq p^+<r^+$. The main results
of this paper read as follows.

\begin{Thm} \label{thm1.1}
Suppose that $1\leq r(x)<p^*(x)$, $r^-<p^-\leq p^+<N$, $Q: \overline{\Omega}\to\R$ is a nonnegative
continuous function and there exists a point $x_1\in\Omega^{-}=\{x\in\Omega\,|\,r(x)<p^-\}$ such that
$Q(x_1)>0$. Then problem \eqref{eq1.2} has infinitely many solutions $\{u_k\}$ with the property
$\|u_k\|_{L^\infty(\Omega)}\rightarrow 0$ as $k\rightarrow \infty$.
\end{Thm}

\begin{Thm} \label{thm1.2}
Suppose that $1<p^-\leq p^+<N$, $1\leq r(x)<p^*(x)$, $r^+>p^+$, $Q: \overline{\Omega}\to\R$ is a
nonnegative continuous function and there exists a point $x_2\in\Omega^{+}=\{x\in\Omega\,|\,r(x)>p^+\}$
such that $Q(x_2)>0$.  Either $r^->p^+$, or $1< r^-\leq p^+$ and there exists $\varepsilon>0$ such that
$Q(x)\equiv 0$ in $\Omega_{\varepsilon}=\{x\in\Omega\,|\, p^--\varepsilon<r(x)< p^++\varepsilon\}$.
Then problem \eqref{eq1.2} has infinitely many  solutions $\{v_k\}$  such that $\|v_k\|\rightarrow \infty$
as $k\rightarrow \infty$.
\end{Thm}

As a corollary of Theorems 1.1 and 1.2, we have

\begin{Cor} \label{cor1.3}
Suppose that $1<p^-\leq p^+<N$, $1\leq r(x)<p^*(x)$, $r^-<p^-$, $r^+>p^+$, $Q(x)$ is a nonnegative
continuous function and there exist $\varepsilon>0$,  $x_1\in\Omega^{-}=\{x\in\Omega\,|\, r(x)<p^-\}$,
$x_2\in\Omega^{+}=\{x\in\Omega\,|\,r(x)>p^+\}$ such that $Q(x_1),\ Q(x_2)>0$ and $Q(x)\equiv 0$ in
$\Omega_{\varepsilon}=\{x\in\Omega\,|\, p^--\varepsilon<r(x)< p^++\varepsilon\}$. Then problem
\eqref{eq1.2} has  infinitely many small solutions $\{u_k\}$ and infinitely many large solutions $\{v_k\}$.
\end{Cor}

In this paper, the letters $C$ and $C_j$ stand for positive constants. $\|u\|_s$ denotes the standard norms
of $L^s(\Omega)\ (s\geq 1)$. The paper is organized as follows. In Section 2, we give some basic properties
of the variable exponent Lebesgue space and Sobolev space. In Sections 3 and 4, we prove Theorems
\ref{thm1.1} and \ref{thm1.2} by the Clark's theorem and the symmetric mountain pass lemma, respectively.

\section{Preliminaries}

\hspace*{\parindent}The variable exponent Lebesgue space $L^{p(x)}(\Omega)$ is defined by
$$L^{p(x)}(\Omega)=\left\{u\ |\ u: \Omega\rightarrow \R\ \text{is measurable},
\ \di_{\Omega}| u|^{p(x)}\,dx<\infty\right\}$$
with the norm
$$|u|_{p(x)}=\inf\left\{\lambda>0\ \Big|\,\di_{\Omega}\left| \df{u}{\lambda}\right|^{p(x)}dx\leq 1\right\}.$$
The variable exponent Sobolev space $W^{1, p(x)}(\Omega)$ is defined by
$$W^{1,p(x)}(\Omega)=\left\{u\in L^{p(x)}(\Omega)\ |\ |\nabla u|\in  L^{p(x)}(\Omega)\right\}$$
with the norm
$$\|u\|_{1,p(x)}=|u|_{p(x)}+|\nabla u|_{p(x)}.$$
Define $W^{1,p(x)}_0(\Omega)$ as the closure of $C^{\infty}_0(\Omega)$ in $W^{1,p(x)}(\Omega)$.
The spaces $L^{p(x)}(\Omega)$, $W^{1,p(x)}(\Omega)$ and $W^{1,p(x)}_0(\Omega)$ are separable
and reflexive Banach spaces if $1<p^-\leq p^+<\infty$ (see \cite{FZ}). Moreover, there is a constant
$C>0$ such that
$$|u|_{p(x)}\leq C|\nabla u|_{p(x)}, \ \ \text{for any}\ u\in W^{1,p(x)}_0(\Omega).$$
Therefore, $\|u\|=|\nabla u|_{p(x)}$ and $\|u\|_{1,p(x)}$  are equivalent norms on $W^{1,p(x)}_0(\Omega)$.
We will use $\|u\|$ to replace $\|u\|_{1,p(x)}$ in the following discussions.

\begin{Lem}  (\cite{FZ})\label{lem2.1}
If $q \in C(\overline{\Omega})$ satisfies $1\leq q(x)<p^{*}(x)$ for $x\in \overline{\Omega}$, then the
imbedding from $W^{1,p(x)}(\Omega)$ to $L^{q(x)}(\Omega)$ is compact and continuous.
\end{Lem}

\begin{Lem} (\cite{FZ, FX3}) \label{lem2.2}
Set
$$\rho(u)=\int_{\Omega}| u|^{p(x)}\,dx,\ \ \mbox{for}\ u\in L^{p(x)}(\Omega).$$
If $u\in L^{p(x)}(\Omega)$ and $\{u_k\}_{k\in\mathbb{N}}\subset L^{p(x)}(\Omega)$, then we have\\
(i) $|u|_{p(x)}<1\ (=1;\ >1)\Leftrightarrow \rho(u)<1\ (=1;\ >1)$;\\
(ii) $|u|_{p(x)}>1\Rightarrow |u|_{p(x)}^{p^-}\leq\rho(u)\leq |u|_{p(x)}^{p^+}$;\\
(iii) $|u|_{p(x)}<1\Rightarrow |u|_{p(x)}^{p^+}\leq\rho(u)\leq |u|_{p(x)}^{p^-}$;\\
(iv) $\lim_{k\rightarrow\infty}|u_k-u|_{p(x)}=0\Leftrightarrow \lim_{k\rightarrow\infty}\rho(u_k-u)=0
\Leftrightarrow u_k\to u\ \mbox{in measure in}\ \Omega\ \mbox{and}\\
\lim_{k\rightarrow\infty}\rho(u_k)=\rho(u)$.
\end{Lem}

Similar to Lemma \ref{lem2.2}, we have

\begin{Lem} \label{lem2.3}
Set $$L(u)=\int_{\Omega}|\nabla u|^{p(x)}\,dx,\ \ \mbox{for}\ u\in W^{1,p(x)}_0(\Omega).$$
If $u\in W^{1,p(x)}_0(\Omega)$ and $\{u_k\}_{k\in\mathbb{N}}\subset W^{1,p(x)}_0(\Omega)$, we have\\
(i) $\|u\|<1\ (=1;\ >1)\Leftrightarrow L(u)<1\ (=1;\ >1)$;\\
(ii) $\|u\|>1\Rightarrow \|u\|^{p^-}\leq L(u)\leq \|u\|^{p^+}$;\\
(iii) $\|u\|<1\Rightarrow \|u\|^{p^+}\leq L(u)\leq \|u\|^{p^-}$;\\
(iv)  $\|u_k\|\rightarrow 0\Leftrightarrow L(u_k)\rightarrow 0$; $\|u_k\|\rightarrow \infty\Leftrightarrow
L(u_k)\rightarrow \infty$.
\end{Lem}

\begin{Def} \label{def2.4}
$u\in W^{1,p(x)}_0(\Omega)$ is called a weak solution of problem \eqref{eq1.2} if
$$\int_{\Omega}|\nabla u|^{p(x)-2}\nabla u\cdot\nabla \phi\,dx=\int_{\Omega}Q(x)| u|^{r(x)-2} u\phi\,dx$$
for all $\phi\in W^{1,p(x)}_0(\Omega)$.
\end{Def}

\section{Infinitely many small solutions}

\hspace*{\parindent} In this section, we use a truncation technique and the Clark's theorem to get a
sequence of solutions converging to zero. We first introduce a variant of the Clark's theorem.

\begin{Thm}  (\cite{LW}, Theorem 1.1)\label{thm3.1}
Let $X$ be a Banach space, $\Phi\in C^1(X,\R)$. Assume $\Phi$ satisfies the Palais-Smale condition
($(PS)$ condition for short), is even and bounded from below, and $\Phi(0)=0$. If for any $k\in \N$, there
exists a $k$-dimensional subspace $X^k$ of $X$ and $\rho_k>0$ such that $\sup_{X^k\cap S_{\rho_k}}\Phi<0$, where $S_{\rho}=\{u\in X\,|\,\|u\|=\rho\}$, then at least one of the following conclusions holds.\\
(i) There exists a sequence of critical points $\{u_k\}$ satisfying $\Phi(u_k)<0$ for all $k$ and
$\|u_k\|\rightarrow 0$ as $k\rightarrow \infty$.\\
(ii) There exists $R>0$ such that for any $0<b<R$ there exists a critical point $u$ such that $\|u\|=b$
and $\Phi(u)=0$.
\end{Thm}

Recall that there is no restriction on $r^+$ Theorem \ref{thm1.1}.  In order to obtain infinitely many small
solutions, we need to have a proper truncation of the nonlinear terms. Let $\phi\in C(\R, \R)$ be an even
function satisfying $0\leq \phi(t)\leq 1$, $\phi(t)=1$ for $|t|\leq \frac{1}{2}$ and $\phi(t)=0$ for $|t|\geq 1$.
Define $g: \overline{\Omega}\times\R\to\R$ by $g(x,t):= Q(x)\phi(t)|t|^{r(x)-2}t$ and consider the auxiliary
problem
\begin{equation}\label{eq3.1}
\left\{\begin{array}{ll}
-\Delta_{p(x)} u=g(x,u), &\mbox{in}\ \Omega,\\
u=0, &\mbox{on}\ \partial \Omega.
\end{array}\right.
\end{equation}
The energy functional $J : W^{1,p(x)}_0(\Omega)\to\R$ associated with \eqref{eq3.1} is defined by
$$J(u)=\di_{\Omega}\df{|\nabla u|^{p(x)}}{p(x)}\,dx-\di_{\Omega}G(x,u)\,dx,$$
where $G(x,t)=\int_0^tg(x,s)\,ds$. We will show that $J$ satisfies the conditions of Theorem 3.1 and obtain
infinitely many solutions $\{u_k\}$ of \eqref{eq3.1} such that $\|u_k\|_{L^\infty(\Omega)}\leq \frac12$ for large
$k$. Then, for large $k$, there holds $g(x,u_k)= Q(x)|u_k|^{r(x)-2}u_k$, and so $u_k$ becomes a solution of \eqref{eq1.2}.

\begin{proof}[\bf Proof of Theorem \ref{thm1.1}.]
From the properties of $\eta$, we see that there exists a constant $M>0$ such that $|g(x,t)|\leq M$ and
$|G(x,t)|\leq M$ for all $(x, t)\in \overline{\Omega}\times \R$. Set $X:=W^{1,p(x)}_0(\Omega)$. Then it is
easy to see that $J(0)=0$, $J\in C^1(X,\R)$ is even and bounded from below, and satisfies the $(PS)$
condition.

Since $r(x_1)<p^-$ and $Q(x_1)>0$, we see from the continuity of $Q$ and $r$ that there exist
$\delta_1>0$, $Q_1>0$ and $r_1<p^-$ such that
\begin{equation}\label{eq3.2}
r^-\leq r(x)< r_1 \ \ \text{and} \ \ Q(x)>Q_1,\ \
\text{for}\ x\in\Omega_1 \triangleq B(x_1,\delta_1)\cap \Omega.
\end{equation}
By the definition of $g$ and \eqref{eq3.2}, we have
\begin{equation}\label{eq3.3}
G(x,u)=\df{Q(x)}{r(x)} |u|^{r(x)}\geq \df{Q_1}{r_1}|u|^{r_1},\ \
\text{for}\ x\in \Omega_1\ \text{and}\ |u|\leq \frac12.
\end{equation}
For $k\in\N$, choose $\{\varphi_j\}_{j=1}^k\subset C_0^\infty (\Omega)$ such that
$$\varphi_j\neq 0,\ \mbox{supp}\,\varphi_j\subset\Omega_1,\ \mbox{supp}\,\varphi_i\cap\mbox{supp}\,\varphi_j=\emptyset\ \mbox{for}\ i\neq j.$$
Let $X^k:=\text{span}\{\varphi_1, \varphi_2, \cdots, \varphi_k\}$. Then $X^k$ is a $k$-dimensional subspace
of $X$. Since any norms in a finite dimensional space are equivalent, there exist $a_k, b_k>0$ such that
\begin{equation}\label{eq3.4}
\|u\|_{r_1}\geq a_k\|u\|,\ \ \|u\|\geq b_k\|u\|_{L^\infty(\Omega)},\ \ \text{for any}\ u\in X^k.
\end{equation}
Set
$$\rho_k=\min\left\{\frac12, \frac{b_k}{2},\left(\frac{p^{-}Q_1a_k^{r_1}}{2r_1}\right)^{\frac{1}{p^{-}-r_1}}\right\}.$$
It follows from \eqref{eq3.3}, \eqref{eq3.4} and Lemma \ref{lem2.3} that, for any $u\in X^k\cap S_{\rho_k}$,
$$J(u)\leq\di_{\Omega}\df{|\nabla u|^{p(x)}}{p(x)}\,dx-\df{Q_1}{r_1}\di_{\Omega_1}|u|^{r_1}dx
\leq\df{1}{p^-}\|u\|^{p^-}-\df{Q_1a_k^{r_1}}{r_1}\|u\|^{r_1}<0.
$$
According to Theorem 3.1, $J$ has a sequence of nontrivial critical points $\{u_k\}$ satisfying $J(u_k)\leq 0$
for all $k$ and $\|u_k\|\rightarrow 0$ as $k\rightarrow \infty$. Since $g(x,t)$ is bounded in
$\overline{\Omega}\times \R$, the weak solutions $\{u_k\}$ belong to $C^{1,\mu}(\overline{\Omega})$ for some
$\mu \in (0,1)$ and they are bounded in this space (see \cite{FX2}). Here $\mu$ is independent of $k$ and $C^{1,\mu}(\overline{\Omega})$ denotes the set of all $C^1(\overline{\Omega})$ functions whose derivatives
are H\"older continuous with exponent $\mu$. Since $C^{1,\mu}(\overline{\Omega})$ is compactly embedded
in $C^1(\overline{\Omega})$, there is a subsequence of $\{u_k\}$, still denoted by itself, such that
$u_k\to u_{\infty}$ in $C^1(\overline{\Omega})$. Since $u_k\rightarrow 0$  in $X$, $u_{\infty}$ must be zero.
By the uniqueness of the limit $u_{\infty}$, we can show that $\{u_k\}$ itself (without extracting a subsequence) converges to zero in $C^1(\overline{\Omega})$. Then $\|u_k\|_{L^\infty(\Omega)}\leq \frac12$ for large $k$ and
so $u_k$ is a solution of \eqref{eq1.2}. The proof is complete.
\end{proof}

\section{Infinitely many large solutions}

\hspace*{\parindent} In this section, we will apply the symmetric mountain pass lemma (see \cite[Theorem 9.12]{PHR}) to get a sequence of large solutions. As in Section 3, we denote $X=W^{1,p(x)}_0(\Omega)$. The energy functional $I : X\to\R$ associated with \eqref{eq1.2} is defined by
$$I(u)=\di_{\Omega}\df{|\nabla u|^{p(x)}}{p(x)}\,dx-\di_{\Omega}\frac{Q(x)}{r(x)}|u|^{r(x)}\,dx.$$
First of all, we prove that the functional $I$ satisfies $(PS)$ condition.

\begin{Lem} \label{lem4.1}
Under the assumption of theorem 1.2, the functional $I$ satisfies $(PS)$ condition.
\end{Lem}
\begin{proof}
Let $\{u_n\}\subset X$  be a $(PS)$ sequence of the functional $I$. Then there exists a constant $C>0$ such
that
\begin{equation}\label{eq4.1}
I(u_n)\leq C,\ \ I'(u_n)\rightarrow 0\ \text{in}\ X^*,
\end{equation}
where $X^*$ denotes the dual space of $X$.

We first prove that $\{u_n\}$ is bounded in $X$. If $r^->p^+$, then it follows from $Q(x)\geq 0$ and Lemma \ref{lem2.3} that
\begin{align*}
r^-I(u_n)-\langle I'(u_n),u_n\rangle
&=\di_{\Omega}\df{r^--p(x)}{p(x)}|\nabla u|^{p(x)}\,dx+\di_{\Omega}\df{r(x)-r^-}{r(x)}Q(x)|u_n|^{r(x)}dx\\
&\geq\df{r^--p^+}{p^+}\min\{\|u_n\|^{p^-}, \|u_n\|^{p^+}\}.
\end{align*}
From \eqref{eq4.1} and $1<p^-\leq p^+$, we see that $\{u_n\}$ is bounded in $X$. If $1< r^-\leq p^+$, we
set $\Omega_{\varepsilon^-}:=\{x\in\Omega\,|\,r(x)\leq p^--\varepsilon\}$ and
$\Omega_{\varepsilon^+}:=\{x\in\Omega\,|\,r(x)\geq p^++\varepsilon\}$. From $Q(x)\geq 0$ and Lemma
\ref{lem2.1}, we have
\begin{align}\label{eq4.2}
\di_{\Omega_{\varepsilon^-}}\df{p^++\varepsilon-r(x)}{r(x)}Q(x)|u_n|^{r(x)}dx
&\leq \df{p^++\varepsilon}{r^-}\sup_{x\in {\Omega}}Q(x)\di_{\Omega}(|u_n|^{p^--\varepsilon}+1)dx\nonumber\\
&\leq C_{\varepsilon}\|u_n\|^{p^--\varepsilon}+C_{\varepsilon}
\end{align}
and
\begin{align}\label{eq4.3}
\di_{\Omega_{\varepsilon^+}}\df{p^++\varepsilon-r(x)}{r(x)}Q(x)|u_n|^{r(x)}dx\leq 0,
\end{align}
where $C_\va>0$. Recall that $Q(x)\equiv 0$ in $\Omega_{\varepsilon}$. By \eqref{eq4.2}, \eqref{eq4.3}
and Lemma \ref{lem2.3}, we have
\begin{align*}
&(p^++\varepsilon)I(u_n)-\langle I'(u_n),u_n\rangle\\
=&\di_{\Omega}\df{p^++\varepsilon-p(x)}{p(x)}|\nabla u|^{p(x)}\,dx
-\di_{\Omega_{\varepsilon^-}}\df{p^++\varepsilon-r(x)}{r(x)}Q(x)|u_n|^{r(x)}dx\\
&-\di_{\Omega_{\varepsilon}}\df{p^++\varepsilon-r(x)}{r(x)}Q(x)|u_n|^{r(x)}dx
-\di_{\Omega_{\varepsilon^+}}\df{p^++\varepsilon-r(x)}{r(x)}Q(x)|u_n|^{r(x)}dx\\\\
\geq&\df{\varepsilon}{p^+}\min\{\|u_n\|^{p^-},\|u_n\|^{p^+}\}
-C_{\varepsilon}\|u_n\|^{p^--\varepsilon}-C_{\varepsilon}.
\end{align*}
Then, by \eqref{eq4.1}, $\{u_n\}$ is bounded in $X$.

Up to a subsequence, we may assume that $u_n\rightharpoonup u$ and then
$\langle I'(u_n),u_n-u\rangle\rightarrow 0$ as $n\rightarrow\infty$. Since the imbedding from X to
$L^{r(x)}(\Omega)$ is compact, we obtain $u_n\rightarrow u$ in $L^{r(x)}(\Omega)$. Then
$$\left |\di_{\Omega}Q(x)|u_n|^{r(x)-2}u_n(u_n-u)dx\right |
\leq\sup\limits_{x\in \Omega}Q(x)\di_{\Omega}|u_n|^{r(x)-1}|u_n-u|dx\to 0.$$
Therefore, one has
$$\di_{\Omega}|\nabla u_n|^{p(x)-2}\nabla u_n\cdot \nabla(u_n-u)dx\to 0.$$
By \cite[Theorem 3.1]{FZ}, we have $u_n\rightarrow u$ in $X$. Therefore, $I$ satisfies $(PS)$ condition.
\end{proof}

Since $r(x_2)>p^+$ and $Q(x_2)>0$, we see from the continuity of $Q$ and $r$ that there exist
$\delta_2>0$, $Q_2>0$ and $r_2>p^+$ such that
\begin{equation}\label{eq4.4}
r_2<r(x)\leq r^+ \ \ \text{and}\ \ Q(x)\geq Q_2,\ \
\text{for all}\ x\in\Omega_2 \triangleq B(x_2,\delta_2)\cap\Omega.
\end{equation}
For $k\in\N$, choose $\{\psi_j\}_{j=1}^k\subset C_0^\infty (\Omega)$ such that
$$\psi_j\neq 0,\ \mbox{supp}\,\psi_j\subset\Omega_2,\
\mbox{supp}\,\psi_i\cap\mbox{supp}\,\psi_j=\emptyset\ \mbox{for}\ i\neq j.$$
Denote $Y^k:=\text{span}\{\psi_1, \psi_2, \cdots, \psi_k\}$.

\begin{Lem}\label{lem4.2}
Under the assumption of theorem 1.2, there exists $R_k>0$ such that
\begin{equation}\label{eq4.5}
I(u)<0,\ \ \text{for any}\ u\in Y^k \ \text{with}\ \|u\|\geq R_k.
\end{equation}
\end{Lem}

\begin{proof}[\bf Proof.]
By \eqref{eq4.4}, we have
$$\frac{Q(x)}{r(x)}|u|^{r(x)}\geq \frac{Q_2}{r^+} |u|^{r_2},\ \ \text{for}\ x\in\Omega_2\ \text{and}\ |u|>1,$$
which implies that
\begin{equation}\label{eq4.6}
\frac{Q(x)}{r(x)}|u|^{r(x)}\geq \frac{Q_2}{r^+}( |u|^{r_2}-1),\ \ \text{for}\ x\in\Omega_2\ \text{and}\  u\in \R.
\end{equation}
Since dim\,$Y^k<\infty$, there exists $\tilde a_k>0$ such that
\begin{equation}\label{eq4.7}
\|u\|_{r_2}\geq \tilde a_k\|u\|,\ \ \text{for any}\ u\in Y^k.
\end{equation}
Using \eqref{eq4.6} and \eqref{eq4.7} we have, for any $u\in Y^k$,
\begin{align*}
I(u)=&\di_{\Omega}\df{|\nabla u|^{p(x)}}{p(x)}\,dx-\di_{\Omega_2}\df{Q(x)}{r(x)}|u|^{r(x)}\,dx\\
\leq&\di_{\Omega}\df{|\nabla u|^{p(x)}}{p(x)}\,dx-\df{Q_2}{r^+}\di_{\Omega_2}(|u|^{r_2}-1)\,dx\\
\leq&\df{1}{p^-}\max\{\|u\|^{p^-},\|u\|^{p^+}\}
-\df{Q_2\tilde a_k^{r_2}}{r^+}\|u\|^{r_2}+\df{Q_2}{r^+}|\Omega_2|.
\end{align*}
Since $r_2>p^+$, \eqref{eq4.5} holds for large $R_k$.
\end{proof}

Define the minimax value
$$c_k=\inf_{h\in G_k}\max_{u\in D_k}I(h(u)),$$
where $D_k=\overline{B}_{R_k}\cap Y^k$ and
$G_k=\{h\in C(D_k, X)\,|\,h \ \text{is odd and}\ h=id \ \text{on}\ \partial B_{R_k}\cap Y^k\}$.

\begin{Rem} \label{rem4.3}
{\rm Using the arguments in the proof of \cite[Lemma 4.9]{KR}, we see that the minimax value $c_k$
is independent of the choice of $R_k$ satisfying \eqref{eq4.5}. Therefore, we can replace $R_k$ by a
larger number such that $\{R_k\}$ is strictly increasing and $\lim_{k\to\infty}R_k=+\infty$.}
\end{Rem}

\begin{Lem} \label{lem4.4}
$c_k\to+\infty$ as $k\to+\infty$.
\end{Lem}

We postpone the proof of Lemma \ref{lem4.4} for a moment and prove Theorem \ref{thm1.2} in the
following.

\begin{proof}[\bf Proof of Theorem \ref{thm1.2}.]
It follows from Lemma \ref{lem4.4} that there exists $k_0\in\N$ and $\alpha>0$ such that
$$c_k\geq \alpha>0,\ \ \mbox{for any}\ k\geq k_0.$$
We claim that, for $k\geq k_0$, the minimax value $c_k$ is a critical value of $I$. If this were false, then,
by Lemma \ref{lem4.1}, there would exist $\va\in(0,\alpha)$ and $\eta\in C([0,1]\times X, X)$ such that
\begin{itemize}
\item $\eta(0, u)=u$ for all $u\in X$;

\item $\eta(1, I^{c_k+\va})\subset I^{c_k-\va}$, where $I^d=\{u\in X\,|\, I(u)\leq d\}$;

\item If $I(u)\not\in[c_k-\va, c_k+\va]$, then $\eta(t,u)=u$ for all $t\in[0,1]$;

\item $\eta(t,u)$ is odd in $u$.
\end{itemize}
Choose $h\in G_k$ such that $\max_{u\in D_k}I(h(u))<c_k+\va$. Then $\eta(1, h(\cdot))\in G_k$ and
$$I(\eta(1, h(u)))\leq c_k-\va,\ \ \mbox{for all}\ u\in D_k.$$
This contradicts the definition of $c_k$.

For $k\geq k_0$, let $v_k$ be a critical point corresponding to $c_k$. Then we have
$$\di_{\Omega}|\nabla v_k|^{p(x)}\,dx=\di_{\Omega}Q(x)|v_k|^{r(x)}dx,$$
which combined with $I(v_k)=c_k$ leads to
\begin{align*}
c_k&=\di_{\Omega}\df{|\nabla v_k|^{p(x)}}{p(x)}\,dx-\di_{\Omega}\df{Q(x)}{r(x)}|v_k|^{r(x)}dx\\
&\leq\df{1}{p^-} \di_{\Omega}|\nabla v_k|^{p(x)}\,dx-\df{1}{r^+}\di_{\Omega}Q(x)|v_k|^{r(x)}dx\\
&=\left(\df{1}{p^-}-\df{1}{r^+}\right)\di_{\Omega}|\nabla v_k|^{p(x)}\,dx\\
&\leq\left(\df{1}{p^-}-\df{1}{r^+}\right)\max\{\|v_k\|^{p^-},\|v_k\|^{p^+}\}.
\end{align*}
By Lemma \ref{lem4.4} and $r^+>p^-$, we have $\|v_k\|\to \infty$ as $k\rightarrow \infty$. The proof
is complete.
\end{proof}

To reach the conclusion, it only remains to prove Lemma \ref{lem4.4}. For this purpose, we recall the
definition and properties of genus which is due to Krasnoselski.

\begin{Def}
Let $E$ be a Banach space. A subset $A$ of $E$ is said to be symmetric if $u\in A$ implies $-u\in A$.
Let $\mathcal{A}$ denote the family of closed symmetric subsets $A$ of $E\setminus\{0\}$. For
$A\in \mathcal{A}$, we define the genus $\gamma(A)$ of $A$ by the smallest integer $m$ such that
there exists an odd continuous map from $A$ to $\R^m\setminus \{0\}$. If there does not exist a finite
such $m$, we define $\gamma(A)=\infty$. Moreover, we set $\gamma(\emptyset)=0$.
\end{Def}

\begin{Lem} \label{lem4.6} Let $A, B\in \mathcal{A}$. Then we have\\
(i) If $A\subset B$, then $\gamma(A)\leq \gamma(B)$.\\
(ii) If there exists an odd continuous map $f\in C(A,B)$, then $\gamma(A)\leq \gamma(B)$.\\
(iii) If $V$ is a bounded symmetric neighborhood of 0 in $\R^N$, then $\gamma(\partial V)=N$. \\
(iv) If $Y$ is a subspace of $E$ such that codim\,$Y=m$ and $\gamma(A)>m$, then $A\cap Y\neq \emptyset$.
\end{Lem}

Similar to \cite[Proposition 9.23]{PHR}, we have

\begin{Lem} \label{lem4.7}
If $Y$ is a closed subspace of $X$ with codim\,$Y<k$, then
$$h(D_k)\cap \partial B_R \cap Y \neq \emptyset,\ \ \text{for all}\ h\in G_k \ \text{and}\ 0<R<R_k,$$
where $G_k=\{h\in C(D_k, X)\,|\,h \ \text{is odd and}\ h=id \ \text{on}\ \partial B_{R_k}\cap Y^k\}$ and
$R_k$ is from Lemma \ref{lem4.2}.
\end{Lem}

\begin{proof}[\bf Proof.]
Set
$V:=\{u\in Y^k\,|\,\|u\|<R_k\ \text{and}\ h(u)\in B_R\}\subset D_k$. Then $0\in V$ and $V$ is bounded
and symmetric in $Y^k$. By Lemma \ref{lem4.6}, we have $\gamma(h(\partial V))\geq\gamma(\partial V)=k$.
Next we claim that
$$\|u\|<R_k,\ \  \mbox{for}\ u\in \partial V.$$
Suppose by contradiction that $\|w_k\|=R_k$ for some $w_k\in \partial V$. Then, since $h=id$ on
$\partial B_{R_k}\cap Y^k$, we have $h(w_k)=w_k$. Hence $R_k=\|w_k\|=\|h(w_k)\|\leq R$, which
contradicts $R<R_k$. Consequently, $\|u\|<R_k$ for $u\in \partial V$. From the definition of $V$ we see
that
\begin{equation}\label{eq4.8}
h(u)\in \partial B_R,\ \text{for}\ u\in \partial V.
\end{equation}
Since codim\,$Y<k\leq \gamma(h(\partial V))$, using Lemma 4.4 yields that
$Y\cap h(\partial V)\neq\emptyset$. Then there is a point $w_0\in \partial V$ such that $h(w_0)\in Y$.
By \eqref{eq4.8}, we have $h(w_0)\in Y\cap \partial B_R$. The proof is complete.
\end{proof}

It is known that there exist $\{e_n\}\subset X$ and $\{f_n\}\subset X^*$ such that
$$X=\overline{\mbox{span}\{e_n\,|\,n=1,2,\cdots\}},\ \ X^*=\overline{\mbox{span}\{f_n\,|\,n=1,2,\cdots\}},$$
and
$$f_n(e_m)=\left\{
\begin{array}{ll}
1, & \mbox{if}\ m=n,\\
0, & \mbox{if}\ m\neq n.
\end{array}\right.$$
For $k\in \N$, we set
$$Y_k=\overline{\mbox{span}\{e_n\,|\, n=k, k+1,\cdots\}},\ \ Z_k=\mbox{span}\{e_n\,|\, n=1, 2,\cdots, k-1\}.$$
Then $X=Y_k+Z_k$ and codim\,$Y_k=k-1$.  Similar to \cite[Lemma 4.1]{SM}, we have the following lemma.

\begin{Lem} \label{lem4.8}
There exists a sequence $\{\delta_k\}$ of positive numbers  such that $\lim_{k\rightarrow\infty}\delta_k=0$
and
$$\|u\|_{r^+}\leq \delta_k\|u\|,\ \ \mbox{for all}\ u\in Y_k.$$
\end{Lem}

Now we are ready to prove Lemma \ref{lem4.4}.

\begin{proof}[\bf Proof of Lemma \ref{lem4.4}.]
Since codim\,$Y_k=k-1$, it follows from Lemma \ref{lem4.7} that
$$h(D_k)\cap \partial B_R \cap Y_k\neq \emptyset, \ \ \text{for all}\ h\in G_k \ \text{and}\ 0<R<R_k.$$
Then
$$\max_{u\in D_k}I(h(u))\geq \inf_{\partial B_R\cap Y_k}I(u),\ \ \text{for all}\ h\in G_k \ \text{and}\ 0<R<R_k,$$
which implies that
\begin{equation}\label{eq4.9}
c_k\geq \inf_{\partial B_R\cap Y_k}I(u),\ \ \text{for all}\ 0<R<R_k.
\end{equation}
By Lemma \ref{lem4.8}, we have
\begin{align*}
I(u)&=\di_{\Omega}\df{|\nabla u|^{p(x)}}{p(x)}\,dx-\di_{\Omega}\df{Q(x)}{r(x)}|u|^{r(x)}\,dx\\
&\geq \df{1}{p^+}\min\{\|u_n\|^{p^-},\|u_n\|^{p^+}\} -C_1\di_{\Omega} (|u|^{r^+}+1)\,dx\\
&\geq \df{1}{p^+}\min\{\|u_n\|^{p^-},\|u_n\|^{p^+}\} -C_2\|u\|^{r^+}_{r^+}-C_3\\
&\geq \df{1}{p^+}\|u_n\|^{p^-} -C_4\delta_k^{r^+}\|u\|^{r^+}-C_5,
\end{align*}
for all $u\in Y_k$.  Combining this with \eqref{eq4.9} leads to
$$c_k\geq\df{1}{p^+}R^{p^-} -C_4\delta_k^{r^+}R^{r^+}-C_5,\ \ \text{for all}\ \ 0<R<R_k.$$
Set $\xi_k=\big(\frac{p^-}{C_4p^+r^+\delta_k^{r^+}}\big)^{\frac{1}{r^+-p^-}}$ and, by Remark \ref{rem4.3},
we may assume that $R_k\geq \xi_k$. Then we have
$$c_k\geq\df{1}{p^+}\xi_k^{p^-} -C_4\delta_k^{r^+}\xi_k^{r^+}-C_5
=\df{r^+-p^-}{p^+r^+}\left(\frac{p^-}{C_4p^+r^+\delta_k^{r^+}}\right)^{\frac{p^-}{r^+-p^-}}-C_5.$$
The desired conclusion follows easily from $\lim_{k\rightarrow\infty}\delta_k=0$ and $r^+>p^-$.
\end{proof}

\noindent{\bf Acknowledgements.}
C. Chu is supported by National Natural Science Foundation of China (No.\,11861021) and Innovation Group Major Program of Guizhou Province (No.\,KY[2016]029). H. Liu is supported National Natural
Science Foundation of China (No.\,11701220).

\end{document}